\documentclass{amsart}
\usepackage{amssymb,amsmath,latexsym,amsthm}%or any usual package
\usepackage[english]{babel}

\newtheorem{theorem}{Theorem}[section]

\newtheorem{proposition}{Proposition}[section]

\theoremstyle{definition}

\numberwithin{equation}{section}

\def\fmod#1 #2{#1\ ({\rm mod}\ #2)}

\newcommand{\ds}{\displaystyle}

\newcommand{\mf}[1]{\mathfrak{#1}}

\renewcommand{\t}[1]{\tau}

\newcommand{\der}{\partial}

\newcommand{\e}{\varepsilon}

\renewcommand{\t}{\tilde }

\newcommand{\w}{\omega}
\renewcommand{\i}{\mf i}

\newcommand\nc{\newcommand}
\nc\on\operatorname
\nc\nop\DeclareMathOperator
\nc\sA{\mathscr A}
\nc{\und}{\underline}
\nc{\derfrac}[2]{\frac{\der{#1}}{\der{#2}}}
\nc{\dv}[1]{\frac{\der}{\der{#1}}}
\nc{\df}{\dfrac}

\renewcommand{\t}{\tau}

\begin{document}

\title{On a conjecture of Dekking : The sum of digits of even numbers}

\author{Iurie Boreico}
\address{Department of Mathematics, Stanford University, 450 Serra Mall, Stanford, California 94305, USA}
\email{boreico@math.stanford.edu}

\author{Daniel El-Baz}
\address{School of Mathematics, University of Bristol, University Walk, Bristol, BS8 1TW, United Kingdom}
\email{Daniel.El-Baz@bristol.ac.uk}

\author{Thomas Stoll}
\address{Universit\'e de Lorraine, Institut Elie Cartan de Lorraine, UMR 7502, Vandoeuvre-l\`es-Nancy, F-54506, France}
\email{thomas.stoll@univ-lorraine.fr}

\maketitle

\begin{abstract}
Let $q\geq 2$ and denote by $s_q$ the sum-of-digits function in base $q$. For $j=0,1,\dots,q-1$ consider
$$\# \{0 \le n < N  : \;\;s_q(2n) \equiv j \pmod q \}.$$ In 1983, F.~M. Dekking conjectured that this quantity is greater than $N/q$ and, respectively, less than $N/q$ for infinitely many $N$, thereby claiming an absence of a drift (or Newman) phenomenon. In this paper we prove his conjecture. 
\end{abstract}

\bigskip
\section{Introduction}

Let $q\geq 2$ and denote by $s_q: \mathbb{N} \to \mathbb{N}$ the sum-of-digits function in the $q$-ary digital representation of integers. In his influential paper from 1968, Gelfond~\cite{gelfond68} proved the following result.\footnote{As usual, we write $f(N)=O(1)$ if $|f(N)|<C$ for some absolute constant $C$, and $f(N)=O_q(1)$ if the implied constant depends on $q$.}
\begin{theorem}
Let $q,d,m\geq 2$ and $a,j$ be integers with $0\leq a<d$ and $0\leq j<m$. If $(m,q-1)=1$ then 
\begin{equation}\label{geleq}
\# \{\;0 \le n < N \, : \;\; n\equiv a\pmod d, \;\;s_q(n) \equiv j \pmod m \} = \frac{N}{dm}+g(N),
\end{equation}
where $g(N)=O_q(N^\lambda)$ with $\lambda = \frac{1}{2\log q}\log \frac{q \;\sin(\pi/2m)}{\sin(\pi/2mq)} < 1$.
\end{theorem}
Shevelev~\cite{shevelev08, shevelev09} recently determined the optimal exponent $\lambda$ in the error term in Gelfond's asymptotic formula when $q=m=2$, and Shparlinski~\cite{shparlinski10} showed that in this case it can be arbitrarily small for sufficiently large primes $d$. 

\medskip

The oscillatory behaviour of the error term $g(N)$ in~(\ref{geleq}) is still not completely understood. The story can be said to have originated with the observation by Moser in the 1960s that for the quintuple of parameters
\begin{equation}\label{newquint}
  (q,a,d,j,m)\equiv (2,0,3,0,2)
\end{equation}  
the error term seems to have \textit{constant} positive sign, \textit{i.e.}, $g(N)>0$ for all $N\geq 1$. In 1969, Newman~\cite{newman69} (with a much more precise result by Coquet~\cite{coquet83}) proved this observation and there is at present a large number of articles which establish so-called \textit{Newman phenomena}, \textit{Newman-like phenomena} or \textit{drifting phenomena} for general classes of quintuples $(q,a,d,j,m)$ extending~(\ref{newquint}). The two main techniques come from a direct inspection of the recurrence relations using the $q$-additivity of the sum-of-digits function, and from the determination of the maximal and minimal value of a related fractal function which is continuous but nowhere differentiable~\cite{goldstein92, coquet83, tenenbaum97}. We refer the reader to the monograph of Allouche and Shallit~\cite{alloucheshallit03} and the article of Drmota and Stoll~\cite{drmotastoll08} for a list of references. Characterizing all $(q,a,d,j,m)$ for which one has a Newman-like phenomenon is still wide open.

\medskip

The aim of the present article is to prove a related conjecture Dekking (see~\cite[``Final Remark'', p. 32-11]{dekking83}) made in 1983 at the S\'eminaire de Th\'eorie des Nombres de Bordeaux concerning a \textit{non-drifting phenomenon}, that is, a situation where the error $g(N)$ is \textit{oscillating in sign} (as $N\to \infty$). To our knowledge, this conjecture has not yet been addressed in the literature, and we will provide a self-contained proof here. 

\bigskip

\noindent \textbf{Conjecture (Dekking, 1983):} Let $q\geq 2$ and $0\leq j<q$ and set
$$(q,a,d,j,m)\equiv  (q,0,2,j,q).$$
Then $g(N)<0$ and $g(N)>0$ infinitely often.

\bigskip

Dekking was mostly interested in finding the optimal error term in~(\ref{geleq}) (or, as he puts it, the \textit{typical exponent} of the error term) and obtained various results for the cases $q=2$, $d$ arbitrary, and $d=2$, $q$ arbitrary. As for the conjecture, he proved the case of $q=3$, $j=0,1,2$ via an argument with a geometrical flavour. 

\medskip

Our main result is as follows. 

\begin{theorem}\label{mtheo}
Let $q\geq 2$, $0\leq j<q$ and set
$$(q,a,d,j,m)\equiv (q,a,d,j,q).$$
\begin{enumerate}
\item[(i)] If $d\mid q$, then $g(N)=O(1)$ and $g(N)$ changes signs infinitely often as $N\to \infty$.
\item[(ii)] If $d\mid q-1$, then $g(N)$ can take arbitrarily large positive values as well as arbitrarily large negative values as $N\to \infty$. 
\end{enumerate}
\end{theorem}

In the case of $d=2$ this proves Dekking's conjecture and covers all bases $q\geq 2$.

\section{Proof of Theorem~\ref{mtheo}}\label{secdekking}

For an integer $n\geq 0$, we write
$$n=(\varepsilon_k,\varepsilon_{k-1},\ldots,\varepsilon_0)_q$$
to refer to its $q$-ary digital expansion $n=\sum_{i=0}^k \varepsilon_i q^i$. Let $U(n) = \{ z \in \mathbb{C} \, | \, z^n = 1 \}$ denote the set of the $n$th roots of unity. 
We will make use of the following well-known formula from discrete Fourier analysis.
\begin{proposition}\label{fourier}
Let $ f(x) = \sum_{k = 0}^{\infty} a_k x^k\in\mathbb{C}[x]$, $n\geq 1$, $l\geq 0$ and set $\omega_n = e^{2 \pi i/ n}$. Then
$$\sum_{k \equiv l \, (\bmod n)} a_k x^k = \frac 1n \sum_{s = 0}^{n-1} \omega_n^{-ls} f(\omega_n^s x).$$
\end{proposition}
\begin{proof}
The coefficient of $x^j$ in $\frac 1n \sum_{s = 0}^{n-1} \omega_n^{-ls} f(\omega_n^s x)$ is $\frac 1n \sum_{s = 0}^{n-1} a_j \omega_n^{s(j - l)},$ that is $a_j$ if $j \equiv l$~(mod $n$) and $0$ otherwise. \end{proof}

We deal with (i) $d\mid q$ and (ii) $d\mid q-1$ in Theorem~\ref{mtheo} separately in the two subsequent sections.

\subsection{The case $d\mid q$}

For $d=2$, $q$ even, Dekking remarked and left to the readers of his article (see~\cite[Remark before Proposition~5, p.32-08]{dekking83}) that the typical exponent $\lambda$ equals 0, \textit{i.e.}, $g(N)=O(1)$. This is due to the fact that when $q$ is even then the parity of an integer is completely encoded in the last digit of its base $q$ expansion. A similar situation applies when $d \mid q$. In order to find the oscillatory behaviour of $g(N),$ we calculate $g(N)$ explicitly.

\medskip

Define $$f_j(n) = c_j(n) - \frac 1q,$$ where 
$$c_j(n) =
\begin{cases}
 1 & \text{if }s_q(n) \equiv j \pmod q; \\
 0 & \text{otherwise.}
\end{cases}$$
Consider 
\begin{equation}\label{Djdef}
  D_j(N) = \sum_{\substack{0\leq n < N \\ n \equiv a \pmod d}} f_j(n),   
\end{equation}
thus
\begin{equation}\label{gDj}
  g(N)=D_j(N)-\frac{N}{dq}+\frac{1}{q}\left\lceil\frac{N-a}{d}\right\rceil.
\end{equation}
We want to find  infinitely many values of $N$ such that $g(N)>0$, respectively, $g(N)<0$. Since an integer in base $q$ (with $q$ divisible by $d$) gives remainder $a$ mod $d$ if and only if its last digit in base $q$ gives remainder $a$ mod $d$, we get for $N=(\varepsilon_k,\ldots,\varepsilon_0)_q$,
\begin{align*}
  D_j(N)&=\sum_{r=2}^{k}\sum_{\delta=0}^{\varepsilon_{r}-1}\; \sum_{\substack{0\leq i_0,i_1,\ldots,i_{r-1}\leq q-1\\i_0 \equiv a \pmod d}}\;f_j((\varepsilon_k,\ldots,\varepsilon_{r+1},\delta,i_{r-1},\ldots,i_0)_q)\\
&\quad +\sum_{\delta=0}^{\varepsilon_{1}-1} \sum_{\substack{i_0=0\\ i_0 \equiv a \pmod d}}^{q-1}f_j((\varepsilon_k,\ldots,\varepsilon_2,\delta,i_0)_q)\\
&\quad +\sum_{\substack{\delta=0\\ \delta \equiv a \pmod d}}^{\varepsilon_{0}-1} f_j((\varepsilon_k,\ldots,\varepsilon_1,\delta)_q).
  \end{align*}
For $r\geq 2$ we get
$$\sum_{\substack{0\leq i_0,i_1,\ldots,i_{r-1}\leq q-1\\i_0 \equiv a \pmod d}}f_j((\varepsilon_k,\ldots,\varepsilon_{r+1},\delta,i_{r-1},\ldots,i_0)_q)=
D_{j-\varepsilon_k-\cdots-\varepsilon_{r+1}-\delta}(q^r)=0.$$
Set $\alpha=j-s_q(N)+\varepsilon_1+\varepsilon_0$ and $\beta= j-s_q(N)+\varepsilon_0$. For the other two terms we then get by a direct calculation,
\begin{equation}\label{form1}
  \sum_{\delta=0}^{\varepsilon_{1}-1} \sum_{\substack{i_0=0\\ i_0 \equiv a \pmod d}}^{q-1}f_j((\varepsilon_k,\ldots,\varepsilon_2,\delta,i_0)_q)=
-\frac{\varepsilon_1 }{d}+\sum_{\delta=0}^{\varepsilon_1-1}\sum_{\substack{0\leq i_0<q\\ i_0\equiv a \pmod d\\i_0\equiv \alpha-\delta \pmod q}}1
\end{equation}
and
\begin{equation}\label{form2}
\sum_{\substack{\delta=0\\ \delta \equiv a \pmod d}}^{\varepsilon_{0}-1} f_j((\varepsilon_k,\ldots,\varepsilon_1,\delta)_q)=
-\frac{1}{q}\left\lceil\frac{\varepsilon_0-a}{d}\right\rceil+\sum_{\substack{0\leq \delta< \varepsilon_0\\ \delta\equiv a \pmod d\\\delta\equiv \beta \pmod q}}1.
\end{equation}
From~(\ref{gDj}),~(\ref{form1}) and~(\ref{form2}) it is straightforward to find sequences of positive integers $N$ with $g(N)>0,$ respectively $g(N)<0$. In fact, if $a\neq 0$ we can take all $N$ with $\varepsilon_1=0$, $\varepsilon_0=a$ to get $g(N)=-\frac{a}{qd}<0$. For $a=0$ we take all $N$ with $\varepsilon_1=1$, $\varepsilon_0=a$ and $s_q(N)\not\equiv j+1 \pmod d$ to get $g(N)=-1/d<0$. On the other hand, if $a+1<q$ we may take all $N$ with $\varepsilon_1=0$, $\varepsilon_0=a+1$ to find $g(N)=1+\frac{1}{d}-\frac{a+1}{qd}-\frac{1}{q}>0$. If $a+1=q$ (which again implies $d=q$) we take all $N$ with $\varepsilon_1=1$, $\varepsilon_0=0$ and $s_q(N)\equiv j+2 \pmod q$ to get $g(N)=-\frac{1}{d}+1>0$. This completes the proof in this case.

\subsection{The case $d\mid q-1$}

In what follows, set
$$E_{a,j}(k)=\# \{\;0 \le n < q^{k} \, : \;\; n\equiv a\pmod d, \;\;s_q(n) \equiv j \pmod q \},$$
where $a, j$ are fixed integers with $0\leq a<d$, $0\leq j<q$ and $k\geq 1$. Consider the generating polynomial in two variables	
$$P(x, y) = \prod_{i=0}^{k-1} (1 + x y^{q^i} + x^2 y^{2q^i} + \cdots + x^{q-1} y^{(q-1) q^i}),$$
which encodes the digits of integers less than $q^k$ in base $q$. Denote by $[x^u y^v] P(x, y)$ the coefficient of $x^u y^v$ in the expansion of $P(x, y)$. 
By Proposition~\ref{fourier},
$$E_{a,j}(k) = \sum_{\substack{u \equiv j \, (\text{mod } q) \\ v \equiv a \, (\text{mod } d)}} [x^u y^v] P(x, y)= \frac 1{dq}\sum_{\substack{\w \in U(q) \\ \e\in U(d)}} \w^{-j}\e^{-a}P(\w, \e).$$

For $\e \in U(d)$ with $d\mid q-1$ we have $\e^{lq^i}=\e^{l}$ for $0\leq l \leq q-1$ and thus
$$P(\w, \e)=(1+\w\e+\w^2\e^2+\ldots+\w^{q-1}\e^{q-1})^k.$$
Since $\w\e=1$ if and only if $\w=\e=1$ ($d$ and $q$ are coprime) and $\w^q\e^q=\e$ we get
\begin{equation}\label{Eaj}
  E_{a,j}(k)-\frac{q^{k-1}}{d}=\frac 1{dq}\sum_{\substack{\w \in U(q) \\ \e\in U(d) \\ \w\e\neq 1}} \w^{-j}\e^{-a} \left(\frac{1-\e}{1-\w\e}\right)^k.
\end{equation}

We now take a closer look at the dominant term on the right hand side in ~(\ref{Eaj}).
Note that for $\w\in U(q), \e\in U(d)$ with $\w \e \neq 1,$ we have
$$\frac{1}{\pi}\;\arg\left(\frac{1-\e}{1-\w\e}\right)\in \mathbb{Q}.$$

We claim that the numbers $\ds \frac{1-\e}{1-\w\e}$ are all pairwise distinct. Indeed, for any point on the unit circle $z\ne 1$, it can easily be seen (geometrically or otherwise) that  $\arg((1-z)^2) = \arg(z)+\pi$. It follows that $$\ds\arg\left(\left(\frac{1-\e}{1-\w\e}\right)^2\right) = - \arg(\w).$$ Therefore, if $$\ds\frac{1-\e}{1-\w\e}=\frac{1-\e'}{1-\w'\e'}$$ then we conclude that $\w$ and $\w'$ have the same argument so $\w=\w'$, and then $\e=\e'$. This means that there are no cancellations in ~(\ref{Eaj}).

Write
$$R=\max \left\{\left|\frac{1-\e}{1-\w\e}\right|:\quad \w \in U(q), \; \e\in U(d),\; \w\e\neq 1\right\}$$
and let $r_1, r_2, \ldots, r_h$ be all of the numbers $(1-\e)/(1-\w\e)$ whose absolute value equals $R$. 

The set $U(d)$ divides the unit circle into $d\geq 2$ equal parts, so it always contains an element $\e_0$ in the open half-plane $\textrm{Re}(\e)<0$. Similarly, $U(q)$ must contain an element $\w_0$ in the closed half-plane $\textrm{Re}(\e_0\w)\geq 0$. Then $|1-\e_0|>\sqrt 2$ while $|1-\w_0\e_0|\leq \sqrt{2}$, thus $$\left|\frac{1-\e_0}{1-\w_0\e_0}\right|>1.$$ Note also that $\w_0\e_0\neq 1$ as $(d,q)=1$ and $\e_0\neq 1$. 

It follows that $R>1$, which in particular implies that the value 1 is not among these $r_i$. Then, as $k\to \infty$,
$$\sum_{\substack{\w \in U(q) \\ \e\in U(d) \\ \w\e\neq 1}} \w^{-j}\e^{-a} \left(\frac{1-\e}{1-\w\e}\right)^k\quad \sim \quad R^k \sum_{i=1}^h c_i \left(\frac{r_i}{R}\right)^k,$$
for certain $c_i \in \mathbb{C}$ which are not all zero. As the $r_i$ all have arguments equal to rational multiples of $\pi$, the $r_i/R$, $i=1,\ldots, h$, are roots of unity. Therefore there exists an integer $M\geq 1$ such that $(r_i/R)^M=1$ for all $i$. Write
$$c'(k)=\sum_{i=1}^h c_i \left(\frac{r_i}{R}\right)^k.$$
Since $E_{a,j}(k)$ is real and $c'(k+M)=c'(k)$ for all $k$ we must have that $c'(k)\in \mathbb{R}$ for all $k$. Moreover,
$$\sum_{k=0}^{M-1} c'(k)=\sum_{i=1}^h c_i \sum_{k=0}^{M-1} \left(\frac{r_i}{R}\right)^k=0,$$
since $r_i$ is not real for all $i$. Thus, among all the $c'(k)$ there is at least one positive and at least one negative value. Let $-c'_1=c'(k_1)<0$ be the smallest negative value and $c_2=c'(k_2)>0$ be the largest positive value among them. 
Then, as $k \to \infty$,
$$E_{a,j}(k)-\frac{q^{k-1}}{d} \sim -\frac{c'_1}{dq}\, R^k<0, \qquad \text{for}\quad k\equiv k_1 \pmod M$$
and
$$E_{a,j}(k)-\frac{q^{k-1}}{d} \sim \frac{c'_2}{dq}\, R^k>0, \qquad \text{for}\quad k\equiv k_2 \pmod M.$$
This completes the proof.

\section{Acknowledgment}
We would like to express our gratitude to J.-P. Allouche, F. M. Dekking and J. Shallit for related correspondence on this problem.

\end{document}